\documentclass[10pt]{article}
\usepackage[utf8]{inputenc}
\usepackage{latexsym,amssymb,amsmath,url,authblk}

\def\N{\mathbb{N}}

\def\R{\mathbb{R}}

\def\proof{\par\noindent{\em Proof. }}
\def\eproof{\hfill{$\Box$}\bigskip}

\def\ds{\dots}
\def\sus{\subset}
\def\al{\alpha}
\def\be{\beta}
\def\ga{\gamma}
\def\de{\delta}

\def\cc{\colon}
\def\ep{\varepsilon}

\newtheorem{thm}{Theorem}[section]
\newtheorem{prop}[thm]{Proposition}

\newtheorem{lem}[thm]{Lemma}

\title{Thomassen's proof and Filippov's proof of the Weak Jordan Theorem}
\author{M. Klazar\footnote{{\tt klazar@kam.mff.cuni.cz}}\ \ 
(KAM MFF UK, Praha)}
\date{\today}

\begin{document}

\maketitle

\begin{abstract}
We present, in detail and with a~modern rigor, the two title proofs. The Weak Jordan 
Theorem (WJT) states that the complement of any topological circuit in the plane is disconnected.
\end{abstract}

\tableofcontents

\section{JT, WJT and AT. Auxiliary results}\label{sec_intro}

{\bf JT, WJT and AT. }Let us begin with a~few necessary definitions. 
We denote by $I$ any compact real interval $I=[a,b]$ with $a<b$. If 
$f\cc I\to X$ is a~map, we set $f[I]=\{f(x)\cc\; x\in I\}$, $f^0=
\{f(x)\cc\; a<x<b\}$ and 
$\mathrm{ep}(f)=\{f(a),f(b)\}$. 
We also say that {\em $f$ joins $f(a)$ to $f(b)$ in $X$}. If $f(a)=f(b)$, we call $f$ {\em closed}. If $g\cc A\to B$ is any map and $C$ is {\em any} set, we define 
$$
g[C]=\{g(x)\cc\;x\in A\cap C\}\,\text{ and }
g^{-1}[C]=\{x\in A\cc\;g(x)\in C\}\,.
$$

An {\em arc} is a~continuous injection $f\cc I\to\R^2$. A~{\em circuit} is 
a~continuous map $f\cc I=[a,b]
\to\R^2$ such that $f(x)=f(y)$ for $x\ne y$ if and 
only if $\{x,y\}=\{a,b\}$. The following {\em Jordan Theorem}, or JT, 
is a~fundamental result in plane topology. It is due to C.~Jordan \cite[pp. 587--594]{jord} in 1887.

\begin{thm}[Jordan, 1887]\label{thm_jordan}
For any circuit $f\cc I\to\R^2$, the complement
$$
\R^2\setminus f[I]=A\cup B
$$
is a~disjoint union of two nonempty open connected sets $A$ and $B$ such that $A$ is bounded, $B$ is unbounded, and $\partial A=\partial B=f[I]$.
\end{thm}
We explain the statement. We work in the {\em Euclidean plane}, which is 
a~metric space
$\langle\R^2,d_2\rangle$ with the 
metric
$$
d_2(a,\,b)=\sqrt{(a_x-b_x)^2+(a_y-b_y)^2}\,,
$$
and in its subspaces. We count among them $\R$ and $I$. A~set $X\sus M$ in a~metric space $\langle M,d\rangle$ is {\em disconnected} if there exist open (equivalently, closed) sets $A,B\sus M$ such that
$$
A\cup B\supset X\;\&\;A\cap X\ne\emptyset\ne B\cap X\;\&\;
A\cap B\cap X=\emptyset
$$
---\,we say that {\em $A$ and $B$ cut $X$}. Else, if such sets $A$ and $B$ do not exist, $X$ is {\em connected}. 
We take for granted that every real interval (not necessarily compact) is connected. A~set 
$X\sus M$ is {\em bounded} if $X\sus B=B(a,r)$ for a~ball $B$ in $\langle 
M,d\rangle$. Else, if $X$ is not contained in any ball, 
$X$ is {\em unbounded}. Finally, the {\em boundary of $X$} 
is the set
$$
\partial X=\{a\in M\cc\;\forall r\, B(a,r)\cap X\ne\emptyset\ne B(a,r)\cap (M\setminus X)\}\,.
$$

In this article, we will not prove the JT, but we present two proofs for
the weaker version, the {\em Weak Jordan Theorem}, or WJT.

\begin{thm}[WJT]\label{thmWJT}
For any circuit $f\cc I\to\R^2$, the complement
$$
\R^2\setminus f[I]
$$    
is disconnected.
\end{thm}
JT $\Rightarrow$ WJT because the disjoint union of two nonempty open 
sets is disconnected.
The proof of the JT is harder than that of the WJT because it requires the following ``lemma'', the {\em Arc Theorem}, or AT. 

\begin{thm}[AT]\label{thm_AT}
For any arc $f\cc I\to\R^2$, the complement
$$
\R^2\setminus f[I]
$$    
is connected.    
\end{thm}
We write about this overlooked cornerstone of the theory of planar 
graphs in a~separate contribution \cite{klaz}, where we also complete 
the presentation of a~proof of the JT.

\medskip\noindent
{\bf Thomassen's proof of the WJT. }In 1992, C.~Thomassen found in \cite{thom} an 
interesting proof 
of the 
JT based on planar graphs. He actually obtained combinatorial 
proofs of the more general 
Jordan--Sch\"onflies theorem and of the theorem on classification of surfaces. In 
Section~\ref{sec_Thomassen} we present only a~part of \cite{thom}, 
Thomassen's proof of the WJT. It employs the following virtual 
configurations. 
A~{\em $K_{3,3}$-configuration} has two disjoint three-element sets 
$U,V\sus\R^2$ and nine 
arcs $f_{u,v}\cc I\to\R^2$ for $u\in U$ and $v\in V$ such that 
$f_{u,v}$ joins $u$ to $v$ and 
$$
f_{u,\,v}^0\cap f_{u',\,v'}^0=\emptyset=f_{u,\,v}^0\cap(U\cup V) 
$$
for every two pairs $\langle u,v\rangle\ne\langle u',v'\rangle$. It 
is a~plane drawing, without crossings, of the  complete bipartite 
graph $K_{3,3}$; these configurations do not exist. A~PL {\em map} is a~map $f\cc [a,b]\to\R^2$ such that 
for a~partition $a=a_0<a_1<\ds<a_k=b$ of $[a,b]$, all restrictions $f\,|\,
[a_{i-1},a_i]$ for $i=1,2,\ds,k$ are non-constant linear maps. Every PL map is continuous. The straight segments 
$f[\,[a_{i-1},a_i]\,]$  
($\sus\R^2$) 
are the {\em segments of $f[I]$}. The points $f(a_i)$ ($\in\R^2$) are 
the {\em corners of $f[I]$}.
In a~PL {\em $K_{3,3}$-configuration}, the nine arcs are PL maps. Thomassen's proof of the WJT in 
\cite{thom}, that we present in Section~\ref{sec_Thomassen}, is the sequence claim~1 $\Rightarrow$ 
claim~2 $\Rightarrow$ claim~3 $\Rightarrow$ claim~4, where the four claims are as follows.
\begin{enumerate}
\item The Intermediate JT holds for PL circuits: the complement of any PL 
circuit is a~disjoint union of two nonempty open connected sets.
\item PL $K_{3,3}$-configurations do not exist.
\item $K_{3,3}$-configurations do not exist.
\item The WJT holds.
\end{enumerate}
Thomassen's main invention is the last implication claim~3 $\Rightarrow$ claim~4. 

\medskip\noindent
{\bf Filippov's proof of the WJT. }Another proof of the WJT is contained in the 
proof of the JT given by A.\,F. Filippov in \cite{fili} in 1950. We 
present this proof  in Section~\ref{sec_Filippov}. We divide it in three steps.
\begin{enumerate}
\item Let $f,g\cc I=[a,b]\to\R^2$ be PL maps such that 
$f[I]\cap g[I]=\emptyset$ and that either $f$ is closed or 
$g[I]$ lies between the two vertical lines going through the points $f(a)$ 
and $f(b)$. Consider the function 
$$
N=N(c)=N(c,\,f)\cc g[I]\to\{0,1\}
$$ such that  
$N(c)$ is the parity of the number of simple intersections of the vertical 
half-line going up from the point $c$ with $f[I]$ (we provide a~precise 
definition later). Then $N$ is constant.

\item The function $N(c,f)$ is additive in the variable $f$. If $f$ is not closed 
and the point $c$ lies between the two vertical lines and below all 
intersections of the line $x=c_x$ with $f[I]$, then $N(c)=1$.

\item Let $f\cc I\to\R^2$ be a~circuit. The results in steps~1 
and~2 are used to obtain two distinct points $c,d\in\R^2\setminus f[I]$ such 
that for every PL map $g$ joining $c$ to $d$ we have $g[I]\cap f[I]\ne\emptyset$. This means that the 
complement $\R^2\setminus f[I]$ is disconnected and the WJT holds.
\end{enumerate}

\noindent
{\bf Why this presentation of the two proofs? }They are beautiful, 
but they are not rigorous by modern standards. Since we like them, we 
decided to 
elevate them to the present level of rigor. See also the discussion in 
Section~\ref{sec_concl}.

\medskip\noindent
{\bf Auxiliary results. }In the rest of Section~\ref{sec_intro} we present, 
with proofs, auxiliary results that will be needed in Sections~\ref{sec_Thomassen} 
and \ref{sec_Filippov}. Since they are mostly standard, the impatient reader may 
now jump to the two proofs and check 
any auxiliary result only if the need arises.

We begin with three constructions related to maps. Let $I=[a,b]$ and $J=[c,d]$ 
be real intervals with $a<b$ and $c<d$, and let $f\cc I\to X$ be 
a~map. We can parametrize 
$f[I]$ by $J$
via the map $g=f(h)\cc J\to X$, where $h\cc J\to I$ is 
the linear homeomorphism $h(x)=\frac{b-a}{d-c}(x-c)+a$. Clearly, $g$ 
joins $g(c)=f(a)$ to $g(d)=f(b)$ in $X$. If $f$ is a PL map, then so is $g$, and 
the same holds for being a continuous map, an arc, or a~circuit. Thus we 
can use any compact interval $I=[\al,\be]$, $\al<\be$, as the 
common definition domain of PL maps, (PL) arcs, and (PL) circuits.

Let $f,g\cc I=[a,b]\to X$ be two maps such that $f(b)=g(a)$. We write 
$f+g$ for their concatenation  $h=f+g\cc I\to X$ 
given by $h(x)=f(2x-a)$ for $a\le x\le \frac{1}{2}(a+b)$, and by $h(x)=g(2x-
b)$ for $\frac{1}{2}(a+b)\le x\le b$. 
Clearly, $h$ joins $h(a)=f(a)$ to $h(b)=g(b)$ in $X$. If $f$ 
and $g$ are PL maps, then so is $h$, and the same holds for continuity. Let 
in addition $f[I]\cap g[\,(a,b]\,]=\emptyset$, and let
$f$ and $g$ be arcs. Then $f+g$ is an arc. Let in addition  $f[I]\cap g^0=\emptyset$, $g(b)=f(a)$, and let $f$ and $g$ be arcs. Then $f+g$ is a~circuit.

Finally, let $U$, $V$, and 
$\{f_{u,v}\cc\;u\in V,v\in V\}$ be a~$K_{3,3}$-configuration. For 
any of the arcs $f_{u,v}\cc I=[a,b]\to\R^2$, we denote by $f_{v,u}\cc 
I\to\R^2$ the reverse arc 
$$
f_{v,\,u}(x)=f_{u,\,v}(a+b-x)\,.
$$

The first two auxiliary results concern closed sets in metric spaces. 
Recall that a~set $X\sus M$ in a~metric space 
$\langle M,d\rangle$ is {\em closed} if the complement $M\setminus X$ is 
an open set, that is, if for every point $a\in M\setminus X$ there exists an $r>0$
such that the ball $B(a,r)\sus X\setminus M$. It is well known that $X\sus M$ is closed if and only if for every sequence $(a_n)\sus X$ with $\lim a_n=a\in M$ we have $a\in X$.

\begin{prop}\label{prop_uniClo}
In any metric space, the union of two closed sets is a~closed set.    
\end{prop}
\proof
Let $X,Y\sus M$ be closed sets in a~metric space $\langle M,d\rangle$, 
and let 
$$
a\in M\setminus(X\cup Y)=(M\setminus X)\cap(M\setminus Y)\,.
$$ 
We have $B(a,r)\sus M\setminus X$ and $B(a,r')\sus M\setminus Y$ for some $r,r'>0$.  Thus 
$$
B(a,\,r'')\sus (M\setminus X)\cap(M\setminus Y)
$$
for $r''=\min(r,r')$ ($>0$). We see that $X\cup Y$ is closed.
\eproof

\begin{prop}\label{prop_imagClos}
Let $\ell\sus\R^2$ be a~vertical half-line going up or down from a~point 
$c\in\R^2$. Then $\ell$ is a~closed set.      
\end{prop}
\proof
Suppose that $\ell$ goes down (if it goes up, we argue similarly) and that 
$a\in\R^2\setminus\ell$. If $a_x=c_x$, we define $r=a_y-c_y$ 
($>0$), and if $a_x\ne c_x$, we set $r=|a_x-c_x|$ ($>0$). Then 
$d_2(a,b)\ge r$ for every point $b\in\ell$, and therefore 
$B(a,r)\sus\R^2\setminus\ell$.
\eproof

\noindent
The use of the positivity of the distance between two disjoint plane 
sets, one 
of which is compact and the other is closed but 
non-compact (for example, a~half-line), is an interesting maneuver in 
Filippov's proof.  

The second group of auxiliary results concerns compactness in metric spaces.
Recall that a~set $X\sus M$ in a~metric space 
$\langle M,d\rangle$ is {\em compact} if every sequence in $X$ has a~convergent 
subsequence with a~limit in $X$. We take for granted the compactness of 
intervals $I=[a,b]$ ($\sus\R^2$).

\begin{prop}\label{prop_imComp}
Let $\langle M,d\rangle$ and $\langle N,e\rangle$ be metric spaces, $f\cc M\to N$ be a~continuous map, and let $X\sus M$ be a~compact set. Then $f[X]$ is a~compact subset of $N$.      
\end{prop}
\proof
Let $(b_n)\sus f[X]$ be a~sequence. Using the axiom of choice, we select 
elements $a_n\in X$ such that $f(a_n)=b_n$. By 
the compactness of $X$, there is a~subsequence $(a_{m_n})$ of $(a_n)$ 
with $\lim a_{m_n}=a\in X$. Since $f$ is continuous,  
$$
\lim b_{m_n}=\lim f(a_{m_n})=f(\lim a_{m_n})=f(a)\in f[X]\,.
$$
\eproof

\begin{prop}\label{prop_CoClBo}
Let $\langle M,d\rangle$ be a~metric space and $X\sus M$ be a~compact set. Then $X$ is a~closed and bounded set.   
\end{prop}
\proof
We show that if $X$ is not closed or not bounded, then it is not compact. 
If $X$ is not closed, then there is a~sequence $(a_n)\sus X$ such that $\lim a_n=a\in M\setminus X$. 
Such $(a_n)$ does not have any subsequence with a~limit in $X$. If 
$X$ is not bounded, then there exists a~sequence $(a_n)\sus X$ such that 
$d(a_m,a_n)\ge1$ for every $m<n$. Such $(a_n)$ does not have any convergent 
subsequence. 
\eproof

\begin{prop}\label{prop_compOpen}
For any continuous map $f\cc I\to\R^2$, the complement
$\R^2\setminus f[I]$ is an open set. 
\end{prop}
\proof
By Proposition~\ref{prop_imComp}, the image $f[I]$ is a~compact subset of 
$\R^2$. By Proposition~\ref{prop_CoClBo}, this 
image is a~closed set.
\eproof

\begin{prop}\label{prop_exitEnter}
Let $\langle M,d\rangle$ be a~metric space, $X\sus M$ be a~closed set, and 
let $f\cc I=[a,b]\to M$ be a~continuous map. 
If $f[I]\cap X\ne\emptyset$, then the subset $f^{-1}[X]$ of $I$ has both 
minimum and maximum. 
\end{prop}
\proof
Since the nonempty set $f^{-1}[X]$ is closed and bounded, $\inf(f^{-1}[X])$ 
and $\sup(f^{-1}[X])$ are elements of it.
\eproof

Recall that two sets $X,Y\sus M$ in a~metric space 
$\langle M,d\rangle$ have distance
$$
d(X,\,Y)=\inf(\{d(x,\,y)\cc\;x\in X,\,y\in Y\})\ \ (\in[0,\,+\infty))\,.
$$
For $X=\{x\}$ we write just $d(x,Y)$ instead of $d(\{x\},Y)$.

\begin{prop}\label{prop_distComSet}
Let $\langle M,d\rangle$ be a~metric space and let
$X,Y\sus M$ be nonempty compact sets that are disjoint. Then $d(X,Y)>0$.    
\end{prop}
We omit the proof because we prove a~stronger result.

\begin{prop}\label{prop_distComSet1}
Let $\langle M,d\rangle$ be a~metric space and let
$X,Y\sus M$ be nonempty sets such that $X$ is compact, $Y$ is closed and $X\cap Y=\emptyset$. Then $d(X,Y)>0$.  
\end{prop}
\proof
For the contrary, let $d(X,Y)=0$. Then there exist 
sequences $(x_n)\sus X$ and $(y_n)\sus Y$ such that $\lim d(x_n,y_n)=0$. 
Since $X$ is compact, we may assume that $\lim x_n=x\in X$. Using the 
triangle inequality, we see that also
$\lim y_n=x$. Since $Y$ is closed, $x\in Y$. Thus $x\in X\cap Y$, which 
contradicts the disjointness of $X$ and $Y$.
\eproof

\noindent
In Thomassen's proof, Proposition~\ref{prop_distComSet} 
suffices. Filippov's proof uses also
Proposition~\ref{prop_distComSet1}.

Recall that a~map $f\cc M\to N$ between metric spaces $\langle 
M,d\rangle$ and $\langle N,e\rangle$ is {\em continuous at a~point $a\in 
M$} if for every $\ep>0$ there is $\de>0$ such that $f[B(a,\de)]\sus 
B(f(a),\ep)$. The map $f$ is {\em continuous} if it is continuous at 
every point in $M$. The map $f$ is {\em uniformly continuous} if for 
every $\ep>0$ there is $\de>0$ such that $d(x,y)\le\de$ $\Rightarrow$ 
$e(f(x),f(y))\le\ep$ for every $x,y\in M$. It is easy to show that for every 
continuous map $f$ and every open (closed) set $X\sus N$ the preimage $f^{-1}[X]$ is 
an open (closed) set in $M$. If $f$ is continuous at $a$ then for 
every sequence $(a_n)\sus M$ with $\lim a_n=a$ we have $\lim 
f(a_n)=f(a)$.

\begin{prop}\label{prop_uniCon}
Let $\langle M,d\rangle$ and $\langle N,e\rangle$ be metric spaces, the 
former being compact, and let $f\cc M\to N$ be a~continuous map. Then $f$ 
is uniformly continuous.    
\end{prop}
\proof
Suppose, for the contrary, that $f$ is not uniformly continuous. Then there 
exist $\ep>0$ and sequences $(a_n),(b_n)\sus M$ such that $\lim 
d(a_n,b_n)=0$, but $e(f(a_n),f(b_n))\ge\ep$ for every 
$n$. Due to the compactness of $M$, we may assume that $\lim a_n=\lim b_n=a$
($\in M$). But $\lim f(a_n)=\lim f(b_n)=f(a)$ ($\in N$) does not hold, in 
contradiction with the continuity of $f$ at $a$.
\eproof

The third group of auxiliary results concerns connectedness.

\begin{prop}\label{prop_conImaCon}
Let $\langle M,d\rangle$ and $\langle N,e\rangle$ be metric spaces, $X\sus M$ be a~connected set and $f\cc M\to N$ be
a~continuous map. Then $f[X]$ is a~connected subset of $N$.
\end{prop}
\proof
We show that if $f[X]$ is disconnected, then so is $X$. Let 
$A,B\sus N$ be open sets cutting $f[X]$. Since $f$ is continuous, the 
subsets $f^{-1}[A]$ and $f^{-1}[B]$ of $M$ are open. It is easy to see that 
they cut $X$.
\eproof

A~graph $G=\langle V,E\rangle$ ($E\sus\binom{V}{2}$) is usually defined as {\em connected} if every two vertices $u,v\in V$ can be joined in 
$G$ by a~path (equivalently, by a~walk). Else $G$ is {\em disconnected}. Recall that a~{\em partition} of a~set $X$ is a~set $Y$ such that 
the elements of $Y$ are nonempty and mutually disjoint sets such that $\bigcup Y=X$.  

\begin{prop}\label{prop_connGr}
A~graph $G=\langle V,E\rangle$ is disconnected if and only if there exists a~partition $\{A,B\}$ of $V$ such that no edge $e\in E$ joins $A$ and $B$. 
\end{prop}
\proof
Suppose that $G$ is disconnected and that $u,v\in V$ cannot be joined in $G$ by any path. We set
$$
A=\{w\in V\cc\;\text{a~path in $G$ joins $u$ to $w$}\}\,\text{ and }\,B=V\setminus A\,.
$$
Then $u\in A$, $v\in B$ and no edge $e\in E$ joins $A$ and $B$. So $\{A,B\}$ is the desired partition of $V$.

Suppose that $A$ and $B$ are as stated, and take any $u\in A$
and $v\in B$. Any path in $G$ joining $u$ and $v$ would contain an edge
$e\in E$ joining $A$ and $B$, which means that there is no such path. We see that $G$ is disconnected. 
\eproof

The {\em intersection graph} corresponding to a~set system $X_v$, $v\in V$, is the graph $\langle V,E\rangle$ 
with the edges $E=\{\{u,v\}\cc\;u,v\in V,u\ne v,X_u\cap X_v\ne\emptyset\}$.

\begin{prop}\label{prop_connByGr}
Let $\langle M,d\rangle$ be a~metric space, let  
$X_v\sus M$, $v\in V$, be a~set system of connected sets 
and let the intersection graph $G=\langle V,E\rangle$ be connected. Then the set
$$
X=\bigcup_{v\in V}X_v\ \ (\sus M)
$$
is connected.
\end{prop}
\proof
Suppose that $X$ is disconnected and that the sets $A,B\sus M$ cut $X$. Every set $X_v$ is completely contained either in $A$ or 
in $B$, and we get the partition $\{V_A,V_B\}$ of $V$ by the sets 
$$
V_A=\{v\in V\cc\;X_v\sus A\}\,\text{ and }\,
V_B=\{v\in V\cc\;X_v\sus B\}\,.
$$
Since $A\cap B\cap X=\emptyset$, there is no edge $e\in E$ joining $V_A$ and $V_B$. Hence $G$ is disconnected by Proposition~\ref{prop_connGr}. 
\eproof

\noindent
We prove, as an application, that balls in the plane are connected.

\begin{prop}\label{prop_ballConn}
Every ball in the metric space $\langle\R^2,d_2\rangle$ is connected. 
\end{prop}
\proof
Let $B=B(b,r)=\{a\in\R^2\cc\;d_2(a,b)<r\}$ be a~ball in the Euclidean 
plane. For any point $a\in 
B\setminus\{b\}$ we denote by $S_a$ the straight 
segment joining $b$ and $a$. Every $S_a$ is connected by 
Proposition~\ref{prop_conImaCon}, as a~continuous image of an interval. 
Thus 
$$
{\textstyle
B=\bigcup_{a\in B\setminus\{b\}}S_a
}
$$
is connected by Proposition~\ref{prop_connByGr} (the intersection graph is complete).
\eproof

\noindent
In general, balls in metric spaces need not be connected.

The following proposition is the most intricate of the auxiliary results. 
We need a~lemma for it.

\begin{lem}\label{lem_isOpen}
Let $X\sus\R^2$ be an open set, let $u,v\in X$ with $u\ne v$, and let there be a~{\em PL} arc joining $u$ to $v$ in $X$. Then there exists 
a~$v$-centered ball $B\sus X$ such that for every point $w\in B$ there 
exists a~{\em PL} arc joining $u$ to $w$ in $X$.     
\end{lem}
\proof
Let $f\cc I=[a,b]\to X$ be a~PL arc 
joining $u$ to $v$ in $X$, let $a=a_0<a_1<\ds<a_k=b$ be the 
partition of $[a,b]$, and let $s_i=f[\,[a_{i-1},a_i]\,]$ for 
$i=1,2,\ds, k$ be the segments of $f[I]$. Since $d_2(s_i,v)>0$ for 
$i=1,2,\ds,k-1$ (Propositions~\ref{prop_imComp} and \ref{prop_distComSet}), we can take small enough $r>0$ such 
that $B=B(v,r)\sus X$ and $s_i\cap B=\emptyset$ for $i=1,2,\ds,k-1$. We 
define for every point $w\in B$ a~PL arc $f_w$ joining $u$ to $w$ in $X$. 
If $w\in B\cap s_k$, then $f_w$ is the initial part of $f$ ending at $w$. If 
$w\in B\setminus s_k$, then $f_w$ is the extension of $f$ by the 
straight segment joining $v$ to $w$. It is easy to see that $f_w$ is as 
claimed.
\eproof

\begin{prop}\label{prop_opArcCon}
Let $X\sus\R^2$ be an open connected set and let $u,v\in X$
with $u\ne v$. Then there exists a~{\em PL} arc joining $u$ to $v$ in $X$. 
\end{prop}
\proof
Let $Y\sus X$ be the set of points $x\in X$ such that $x=u$ or a~PL arc joins $u$ 
to $x$ in $X$. Clearly, $Y\ne\emptyset$ ($u\in Y$) and $Y$ is open (Lemma~\ref{lem_isOpen}). We 
show that also the set $Z=X\setminus Y$ is open. If $Z\ne\emptyset$, we 
have the contradiction that $X$ is disconnected. Thus $Z=\emptyset$ and 
$Y=X$, which implies that there exists a~PL arc joining $u$ to $v$ in 
$X$.

Let $y\in Z$. We take a~ball $B=B(y,r)\sus X$ such that $u\not\in 
B$, which is possible as $y\ne u$, and show that $B\sus Z$. Let $z\in B$ 
with $z\ne y$. Clearly, $z\ne u$. Suppose that $f\cc I\to X$ is a~PL arc joining $u$ to $z$ in $X$.
Let $s=yz$ ($\sus B$) be the straight segment joining $y$ and 
$z$. Considering the first intersection in $f[I]\cap s$ 
(Proposition~\ref{prop_exitEnter}), 
we obtain a~PL arc joining $u$ to $y$ in 
$X$, but this contradicts $y\in Z$. Thus $f$ does not exist and $z\in Z$. 
We see that $B\sus Z$. The set $Z$ is open. 
\eproof

The last group of auxiliary results is related to Thomassen's construction of 
a plane drawing of $K_{3,3}$ based on a~circuit whose complement is connected. 

\begin{prop}\label{prop_higLow}
Let $X\sus\R^2$ be a~compact set. Then there exist four points 
$u,v,r,s\in X$ such that $u_x$ and $v_x$ have, respectively, the maximum 
and minimum $x$-coordinate among all points in $X$, and $r$ and $s$ have 
an analogous property in the $y$-coordinate.    
\end{prop}
\proof
We prove the existence of $v$ (the existence of the other three points 
is established similarly). Since $X$ is bounded (Proposition~\ref{prop_CoClBo}), there exists 
a~sequence $(a_n)\sus X$ such that
$$
\lim_{n\to\infty}(a_n)_x=\inf(\{a_x\cc\;a\in X\})=:\al\,.
$$
Since $X$ is compact, we 
may assume that $\lim a_n=v\in X$.
Since the function $c\mapsto c_x$ is continuous, 
$$
v_x=\big(\lim_{n\to\infty}a_n\big)_x
=\lim_{n\to\infty}(a_n)_x=\al
$$
and $v$ is a~leftmost point of $X$.
\eproof

\begin{prop}\label{prop_CircNoSe}
Let $f\cc I=[a,b]\to\R^2$ be a~circuit. Then there exist four points $u,v,r,s\in f[I]$ such that $u_x<v_x$ and $r_y<s_y$.    
\end{prop}
\proof
We prove the result for $u$ and $v$; the argument for $r$ and $s$ is 
similar. If $u$ and $v$ did not exist, then $f[I]$ would be a~vertical 
straight segment $\{e\}\times[c,d]$.
Let $f(a)=f(b)=\langle e,\al\rangle$ and 
$f(\frac{1}{2}(a+b))=
\langle e,\be\rangle$. Then $\al\ne\be$, for example $\al>\be$. By 
the intermediate value theorem, for every $\ga\in(\be,\al)$ there exist numbers $x\in(a,\frac{1}{2}(a+b))$ and
$y\in(\frac{1}{2}(a+b),b)$ such that 
$$
f(x)=f(y)=\langle e,\,\ga\rangle\,.
$$
This contradicts the almost injectivity of $f$.
\eproof

\begin{prop}\label{prop_delCykl}
Let $f\cc I\to\R^2$ be a~circuit and $u,v\in f[I]$ with $u\ne v$. Then there exist arcs $g,h\cc I\to\R^2$ such that 
$g^0\cap h^0=\emptyset$, $g[I]\cup h[I]=f[I]$, and  
$\mathrm{ep}(g)=\mathrm{ep}(h)=\{u,v\}$.
\end{prop}
\proof
Wlog $f(a)=u$ and $f(b)=v$, where $a<b$ are in $I=[c,d]$. We define $g=f\,|\,[a,b]$ and $h\cc[b,d-c+a]\to\R^2$ by 
$$
\text{$h(x)=f(x)$ for $b\le x\le d$ and 
$h(x)=f(x+c-d)$ for $d\le x\le d-c+a$}\,.
$$
According to the initial remark, we may redefine $g$ and $h$ on $I$. It 
is easy to see that $g$ and $h$ are as stated. 
\eproof

\begin{prop}\label{prop_interValu}
Let $f\cc I=[a,b]\to\R^2$ be continuous, $f(a)_y<f(b)_y$, and let $\ell$ be 
a~horizontal line with the $y$-coordinate in $(f(a)_y,f(b)_y)$. Then 
$\ell\cap f[I]\ne\emptyset$.
\end{prop}
\proof
We apply the intermediate value theorem to the continuous function $f(t)_y\cc I\to\R$. 
\eproof

\begin{prop}\label{prop_onT}
Let $A,B\sus[c,d]$ be nonempty closed sets that are disjoint. Then there exists a~subinterval
$[x,y]\sus[c,d]$ such that $x\in A$ and $y\in B$, or $x\in B$ and 
$y\in A$, and $(x,y)\cap(A\cup B)=\emptyset$.    
\end{prop}
\proof
Wlog we have numbers $a<b$ such that $a\in A$ and $b\in B$. We define $x=\sup(\{\al\in[a,\,b]\cc\;\al\in A\})$. 
Then $x\in A$ and $x<b$. We define $y=\inf(\{\al\in[x,\,b]\cc\;\al\in B\})$. Then $y\in B$ and $y>x$. Clearly, 
$(x,y)\cap(A\cup B)=\emptyset$. Thus $[x,y]$ is the 
desired subinterval.
\eproof

\section{Thomassen's proof of the WJT}\label{sec_Thomassen}

{\bf Claim~1. The Intermediate JT holds for PL circuits. }Suppose that $f\cc I\to\R^2$ is a~PL circuit. We first 
prove that the open set $D=\R^2\setminus f[I]$ (Proposition~\ref{prop_compOpen}) is 
disconnected: we define a~continuous map $g\cc D\to\R$ such that $g[D]=
\{0,1\}$. By Proposition~\ref{prop_conImaCon}, the set $D$ is disconnected because 
$\{0,1\}$ is a~disconnected subset of $\R$. 

In order to define $g$ (we slightly simplify \cite{thom}), we take any point 
$a\in D$, draw a~vertical half-line $\ell(a)$ up from $a$, and consider the 
geometric intersection $P(a)=\ell(a)\cap f[I]$. If $P(a)=\emptyset$, we 
set $g(a)=0$. Else 
$$
P(a)=(z_1<z_2<\ds<z_k),\ 
k\ge1\,,
$$ 
where $<$ is the vertical order and every $z_i$ is either a~union of 
consecutive (vertical) segments of $f[I]$ or a~point in $f[I]$. The $z_i$ are 
mutually disjoint. We call $z_i$ a~{\em simple intersection} if locally 
near $z_i$, the set $f[I]$ lies on both sides of $\ell(a)$. Else, if 
locally $f[I]$ lies on only one side of $\ell(a)$, we call $z_i$ a~{\em 
double intersection}. We define $g(a)\in\{0,1\}$ as
$$
\text{$g(a)=$ the parity of the number of simple intersections $z_i$}\,.
$$

We show that $g$ is continuous. Let $a\in D$. We take a~ball $B=B(a,r)\sus D$
with radius $r>0$ so small that no half-line $\ell(a')$ with $a'\in B$ 
contains a~corner of $f[I]$ not present in $P(a)$. If we move $a$ to
$a'\in B$, simple intersections in $P(a)$ are preserved, 
and every double intersection either does not change (when $\ell(a')\sus\ell(a)$ or 
$\ell(a)\sus\ell(a')$) or disappears or decays into two simple (point) intersections. 
Importantly, no new intersections besides those born from $P(a)$ are introduced. It follows that the number of simple intersections increases by 
an even number in $\N_0=\{0,1,\ds\}$, and therefore $g(a)=g(a')$. The 
function $g$ is continuous at $a$ and on $D$.

Since $f[I]$ is a~union of finitely many straight segments, it 
follows that $f[I]$ has a~highest point, which is a~corner. 
Considering the values of $g$ near this highest corner, we see that 
$g[D]=\{0,1\}$. This concludes the proof that $D$ is disconnected.

We prove that
$$
D=\R^2\setminus f[I]=L(f)\cup R(f)
$$
for two nonempty open connected and disjoint sets $L(f)$ and $R(f)$. 
To define them, we consider the partition $a=a_0<a_1<\ds<a_k=b$ of the domain $I=[a,b]$ of $f$ and list the 
corners $c_i=f(a_i)$ and segments 
$s_i=f[\,[a_{i-1},a_i]\,]$ of $f[I]$ for $i=1,2,\ds,k$ in cyclic 
(indices are taken modulo $k$) order
$$
c_0,\,s_1,\,c_1,\,s_2,\,c_2,\,s_3,\,\ds,\,
c_{k-1},\,s_k,\,c_k=c_0\,.
$$
We fix one of the two orientations of $f[I]$ and define $L(f)$ as the set of points $c\in D$ that can be joined to $f[I]$ in $D$ from the left. It means that there 
is a~PL map $f_c\cc I=[a,b]\to\R^2$ such that $f_c(a)=c$, $f_c(b)\in f[I]$, $f_c^0\sus D$, and the last segment of 
$f_c[I]$ approaches $f[I]$ from the left side with respect to the fixed 
orientation. Note that $f_c^0\sus L(f)$. We similarly define $R(f)$, by joining $c\in D$ to $f[I]$ in $D$ from the right. 
It is easy to see that the sets $L(f)$ and $R(f)$ are nonempty and  
open (Lemma~\ref{lem_isOpen} and its proof), and that $L(f)\cup R(f)=D$. If they are 
connected, it follows that they are disjoint: if $L(f)\cap R(f)\ne\emptyset$, then $L(f)\cup R(f)=D$ is connected by
Proposition~\ref{prop_connByGr}, which contradicts the previous result.

So it remains to prove that $L(f)$ and $R(f)$ are connected sets. We 
prove that $L(f)$ is connected (the argument for $R(f)$ is similar). We 
fix a~point $c\in L(f)$ and for every point
$d\in L(f)$ we define a~connected set $C_d$ such that $c,d\in C_d\sus L(f)$. Then
$$
{\textstyle
L(f)=\bigcup_{d\in D}C_d
}
$$
is connected by Proposition~\ref{prop_connByGr} (the intersection graph is complete).

We define sets $C_d$. For  $i=1,2,\ds,k$, we denote by $\ell_i$ ($\ni c_i$) the 
axis of the two angles determined by $s_i$ and $s_{i+1}$. For $\de>0$ let $c_i(\de)\in\ell_i$ be the point on $\ell_i$ lying
in a distance of $\de$ from $c_i$ on the left of $f[I]$.
Let $s_i(\de)$ be the straight segment joining
$c_{i-1}(\de)$ and $c_i(\de)$. If $\de$ is sufficiently small, then
$$
s_i(\de)\cap s_j=\emptyset\,\text{ for }\,j=i-1,\,i,\,i+1\,.
$$
Also, $d_2(c,s_i)\le\de$ for every $c\in s_i(\de)$. It follows (Propositions~\ref{prop_imComp} and \ref{prop_distComSet}) that there is a~$\de_0>0$ such that $s_i(\de)\sus L(f)$ for 
every $\de\le\de_0$. Let 
$$
{\textstyle
C(\de)=\bigcup_{i=1}^k s_i(\de)\,.
}
$$
Then $C(\de)\sus L(f)$ for every sufficiently small $\de$.
Every set $C(\de)$ is connected by Proposition~\ref{prop_conImaCon} (which implies that every segment 
$s_i(\de)$ is connected) and Proposition~\ref{prop_connByGr} (the intersection graph spans a~$k$-cycle). Let $c,d\in L(f)$ be 
the above mentioned points, and let $f_c$ and $f_d$ be respective PL maps that join them to $f[I]$ in $D$ from the left. We take 
small enough $\de$ such that $C(\de)\sus L(f)$ and that 
$$
f_c^0\cap C(\de)\ne\emptyset\ne f_d^0\cap C(\de)\,.
$$
It follows that
$$
C_d=\{c\}\cup f_c^0\cup C(\de)\cup f_d^0\cup\{d\}\ \ (\sus L(f))
$$
is the desired set. It is connected by
Propositions~\ref{prop_conImaCon} and \ref{prop_connByGr}.

\medskip\noindent
{\bf Claim~2. PL $K_{3,3}$-configurations do not exist. }We assume that $U$, $V$ ($\sus\R^2$) and 
$\{f_{u,v}\cc\;u\in U,v\in V\}$ is a~PL $K_{3,3}$-configuration and obtain a~contradiction. Let 
$U=\{u_1,u_3,u_5\}$ and $V=
\{u_2,u_4,u_6\}$. We write $f_{i,j}$
instead of $f_{u_i,u_j}$.
We consider the PL circuit $C\cc I\to\R^2$ given by the $6$-cycle
$$
C=((((f_{1,\,2}+f_{2,\,3})+f_{3,\,4})+f_{4,\,5})+f_{5,\,6})+f_{6,\,1}\,.
$$
We fix one of the two orientations of $C[I]$ and consider the partition $D=\R^2\setminus C[I]=L(C)\cup 
R(C)$ of $D$ defined in Claim~1. Let
$$
e_1=f_{1,\,4},\ e_2=f_{2,\,5}\,
\text{ and }\,e_3=f_{3,\,6}
$$
be the three remaining arcs of the $K_{3,3}$-configuration. Each set $e_i^0$ 
lies completely either in $L(C)$ or in $R(C)$ ($e_i^0$ is 
connected by Proposition~\ref{prop_conImaCon}). 
Two of them, for example, $e_2^0$ and $e_3^0$, lie in one set, 
for example, $R(C)$. We consider the PL circuit $C'\cc I\to\R^2$ given by the $4$-cycle
$$
C'=((f_{2,\,3}+f_{3,\,4})+f_{4,\,5})+f_{5,\,2}\,.
$$
We orient $C'[I]$ in accordance with the chosen orientation of $C[I]$ (the 
first three arcs of $C'$ already have this orientation). Let 
$g=e_3+f_{6,5}$. We consider the partition
$$
D'=\R^2\setminus C'[I]=L(C')\cup R(C')\,.
$$
Near the corner $u_5$, the set $g^0$ ($\sus D'$) intersects $L(C')$. Near 
the corner $u_3$, $g^0$ intersects $R(C')$. Hence $L(C')$ and 
$R(C')$ cut $g^0$, which is a~contradiction because $g^0$ is 
connected by Proposition~\ref{prop_conImaCon}. In each of the  five other cases 
(when two $e_i^0$ lie in $L(C')$ or in $R(C')$) we obtain a~similar 
contradiction.

\medskip\noindent
{\bf Claim~3. $K_{3,3}$-configurations do not exist. }We assume that 
$U$, $V$, and 
$\{f_{u,v}\cc\;u\in U,v\in V\}$ is a~$K_{3,3}$-configuration. We 
transform it in a~PL $K_{3,3}$-configuration, which by Claim~2 is a~contradiction. Let $U=\{u_1, u_2,u_3\}$ and 
$V=\{u_4,u_5,u_6\}$. 
We again write 
$f_{i,j}$ instead of $f_{u_i,u_j}$. For $i=1,2,\ds,6$, let 
$$
B_i=\overline{B}(u_i,\,r)=\{x\in\R^2\cc\;d_2(x,\,u_i)\le r\}
$$ 
be mutually disjoint closed balls with the radius 
$r>0$ so small that $f_{i,j}[I]\cap B_l=\emptyset$ whenever $l\ne i,j$ (Propositions~\ref{prop_imComp} and \ref{prop_distComSet}). 
Let $D=\R^2\setminus\bigcup_{i=1}^6 B_i$. Using the arcs $f_{i,j}$ and Proposition~\ref{prop_exitEnter}, we 
obtain nine subarcs $g_{i,j}\cc I=[a,b]\to\R^2$ of $f_{i,j}$ such that $g_{i,j}$ joins
a~point $\al_{i,j}\in\partial B_i$ to a~point $\be_{i,j}\in\partial B_j$, that $g_{i,j}^0\sus D$ and that the nine 
sets $g_{i,j}[I]$ are mutually disjoint. Using 
Propositions~\ref{prop_imComp}, \ref{prop_distComSet}, 
and \ref{prop_uniCon}, we obtain $\de>0$ and nine partitions 
of $[a,b]$, 
$$
a=a_{0,\,i,\,j}<a_{1,\,i,\,j}<\ds<a_{k(i,\,j),\,i,\,j}=b\ \ (k(i,\,j)\in\N=\{1,2,\ds\})\,,
$$
such that 
$$
g_{i,\,j}(a_{l-1,\,i,\,j})\in B(g_{i,\,j}(a_{l,\,i,\,j}),\,\de)\,\text{ for }\,l=1,\,2,\,\ds,\,k(i,\,j)\,,
$$
and that
$$
B(g_{i,\,j}(a_{l,\,i,\,j}),\,\de)\cap B(g_{i',\,j'}(a_{l',\,i',\,j'}),\,\de)=
\emptyset
$$
for every $l=0,1,\ds,k(i,j)$, $l'=0,1,\ds,k(i',j')$ and pairs 
$\langle i,j\rangle\ne\langle i',j'\rangle$. Let
$$
{\textstyle
U_{i,\,j}=\bigcup_{l=0}^{k(i,\,j)}B(g_{i,\,j}(a_{l,\,i,\,j}),\,\de)\,.
}
$$
Each of these nine open sets $U_{i,j}$ is connected by Propositions~\ref{prop_connByGr} and 
\ref{prop_ballConn} (the intersection graph spans a~path), and they are mutually disjoint. Using 
Proposition~\ref{prop_opArcCon}, we obtain nine PL arcs 
$h_{i,j}\cc[a,b]\to U_{i,j}$ such that $h_{i,j}$ joins $\al_{i,j}=h_{i,j}(a)=g_{i,j}(a)$ ($\in\partial B_i$) to 
$\be_{i,j}=h_{i,j}(b)=g_{i,j}(b)$ ($\in\partial B_j$) in $U_{i,j}$. We consider the first intersections of 
the straight segments 
$u_i\,\al_{i,j}$ and $u_j\,\be_{i,j}$
with $h_{i,j}[I]$ (Proposition~\ref{prop_exitEnter})
and obtain nine PL arcs
$p_{i,j}$ such that $p_{i,j}$ joins $u_i$ to $u_j$ and 
$$
p_{i,\,j}^0\cap p_{i',\,j'}^0=\emptyset=p_{i,\,j}^0\cap(U\cup V)
$$
for every pairs $\langle i,j\rangle\ne\langle i',j'\rangle$.
Then $U$, $V$, and $\{p_{i,j}\cc\;1\le i\le3,\,4\le j\le6\}$ is a~PL $K_{3,3}$-configuration.

\medskip\noindent
{\bf Claim~4. The WJT holds. }We assume the contrary that there exists 
a~circuit $f\cc I\to\R^2$ such that the complement $\R^2\setminus f[I]$ is 
connected, and obtain from $f$ a~$K_{3,3}$-configuration. By Claim~3 it is a~contradiction.

So let $f\cc I\to\R^2$ be a~circuit with the property that the complement 
$D=\R^2\setminus f[I]$ is connected.
Let $a$ and $b$ be points in $f[I]$ with the maximum and minimum 
$y$-coordinate, respectively (Propositions~\ref{prop_imComp} and \ref{prop_higLow}).
Clearly, $a_y>b_y$ (Proposition~\ref{prop_CircNoSe}).
By Proposition~\ref{prop_delCykl} the points $a$ and $b$ divide $f[I]$ in 
two arcs $f_0,f_1\cc I\to\R^2$ such that $f_0^0\cap f_1^0=\emptyset$, $f_0[I]\cup f_1[I]=f[I]$, and 
$\mathrm{ep}(f_0)=\mathrm{ep}(f_1)
=\{a,b\}$. Let $f_2\cc I\to\R^2$ be a~PL arc joining $a$ to $b$ such that 
$f_2^0\sus D$. For example, $f_2$ starts with a~vertical segment going 
up from $a$, then goes horizontally  sufficiently far to the right, goes vertically down to 
a~level below $b_y$, continues horizontally left just below 
$b$, and closes by going vertically up to $b$. Such a~PL arc $f_2$ exists by Propositions~\ref{prop_imComp} and 
\ref{prop_CoClBo}. Let $\ell$ be any horizontal line with $y$-coordinate in $(b_y,a_y)$. By Propositions~\ref{prop_interValu} and \ref{prop_onT}, the line $\ell$ contains a~straight segment $T$
such that one endpoint $c$ is in $f_0^0$, the other endpoint $d$ is in 
$f_1^0$ and $T^0\sus D$. By the connectedness of $D$ and by 
Propositions~\ref{prop_compOpen}, \ref{prop_exitEnter} and \ref{prop_opArcCon}, there exists an arc $f_{e,g}$ 
joining a~point $e\in T^0$ to a~point $g\in f_2^0$ in $D$ and  such that $f_{e,g}^0\cap(T\cup f_2[I])=\emptyset$. We
set $U=\{a,b,e\}$, $V=\{c,d,g\}$ and define the connecting arcs. The arc 
$f_{e,g}$ is already defined. The arcs
$f_{e,c}$ and $f_{e,d}$ are the corresponding subsegments of $T$. The arcs $f_{a,c}$ and $f_{a,d}$ are the corresponding 
subarcs of $f_0$ and $f_1$, respectively. The arc 
$f_{a,g}$ is the corresponding subarc of $f_2$. The arcs $f_{b,c}$, $f_{b,d}$, and $f_{b,g}$ are defined similarly to 
those starting at $a$. It is easy to check that the nine arcs defined have mutually disjoint interiors and that  
$U$, $V$, and 
$\{f_{u,v}\cc\;u\in U,v\in V\}$ is a~$K_{3,3}$-configuration.

\section{Filippov's proof of the WJT}\label{sec_Filippov}

{\bf Step~1. The parity map $N$ is constant. }Let $f$ and $g$ be as described in Section~\ref{sec_intro}: $f,g\cc I=[a,b]\to\R^2$ are PL maps such that 
$f[I]\cap g[I]=\emptyset$ and that either $f$ is closed or 
$g[I]$ lies between two vertical lines going through the points $f(a)$ 
and $f(b)$. Let $S(f)=\R^2\setminus f[I]$ if $f$ is 
closed, and otherwise let $S(f)$ be the intersection of $\R^2\setminus f[I]$ 
with the open strip between the two lines.

We define the {\em parity map} 
$$
N=N(c)=N(c,\,f)\cc S(f)\to\{0,\,1\}\ \ (\sus\R)\,. 
$$
Let $c\in S(f)$. As in Claim~1 of Thomassen's proof, we denote by $\ell(c)$ the half-line going up from $c$ 
and consider the geometric intersection $P(c)$ of $\ell(c)$ and 
$f[I]$. For $P(c)=\emptyset$ we set $N(c)=0$. Let $P(c)\ne\emptyset$. Since now $f$ need not 
be injective, we consider instead of $\ell(c)\cap f[I]$ the preimage 
of this set. Let $a=a_0<a_1<\ds<a_k=b$ be the partition 
of the domain $[a,b]$ of $f$. For every $i=1,2,\ds,k$, the segment 
$s_i=f[\,[a_{i-1},a_i]\,]$ of $f[I]$ is either disjoint 
to $\ell(c)$ or 
$s_i\cap\ell(c)=\{p\}$ for a~single point $p$ or $s_i\sus\ell(c)$. It 
follows that
$$
P(c)=P(c,\,f):=f^{-1}[\ell(c)\cap f[I]\,]\ \ (\sus[a,\,b])
$$
is a~finite disjoint union of 
components $z$ such that every $z$ is either a~subinterval $[a_i,a_j]$ ($\sus[a,b]$) with $0\le 
i<j\le k$, where for non-closed $f$ the first and last inequality is
strict, or a~point in $(a,b)$. We call $z$ a~{\em simple component} if the values $f(t)$ for $t$ immediately before and after $z$ lie on different 
sides of $\ell(c)$. Else, if they lie on the same side of $\ell(c)$, we call 
$z$ a~{\em double component}. We define
$$
\text{$N(c)=$ the parity of the number of simple components of $P(c)$}\ \ (\in\{0,\,1\})\,.
$$

We prove that the parity map $N$ is continuous on $S(f)$. Let $c\in S(f)$. We take 
a~$\de>0$ small enough such that the ball $B=B(c,\de)\sus S(f)$ (Proposition~\ref{prop_compOpen}) and that no 
vertical line that intersects $B$ contains a~corner of $f[I]$ 
distinct from those whose preimages appear in $P(c)$. Let $c'\in B$ with $c'\ne c$. 
As we move $c$ to $c'$, every simple component of $P(c)$ is preserved
and every double component either does not change or disappears or 
decays in two 
simple (point) components. It follows that $N(c)=N(c')$. 
So $N$ is continuous at $c$ and on $S(f)$. We obtain the main result of 
Step~1:
$$
\text{the parity map $N$ is constant on the set $g[I]$}\,.
$$
Indeed, since $g[I]\sus S(f)$, the map $N(g)\cc I\to\R$ is continuous and $N(g)[I]\sus\{0,1\}$. If $N(g)[I]=
\{0,1\}$, we contradict Proposition~\ref{prop_conImaCon}.

\medskip\noindent
{\bf Step~2. Two properties of the parity map $N$. }The first property is
the additivity of $N(c,f)$ in the variable $f$. Let $f_1,f_2\cc 
I=[a,b]\to\R^2$ be two non-closed PL maps such that $f_1(b)=f_2(a)$. Let $f_3=f_1+f_2$ and let
$c\in S(f_1)\cap S(f_2)\cap S(f_3)$. Then
$$
N(c,\,f_3)=N(c,\,f_1)+N(c,\,f_2)\ \ 
(\mathrm{mod}\;2)\,.
$$
Here is an outline of the proof of this equality: if $S_i$ is the set of simple components in $P(c,f_i)$, 
$i=1,2,3$, then there is a~bijection
$$
\be\cc S_3\to(S_1\times
\{0\})\cup(S_2\times\{1\})\,.
$$ 
We leave the details of $\be$ to the reader as an exercise.

We proceed to the second property of $N$. Let $f\cc I=[a,b]\to\R^2$ be a~non-closed PL map 
and let $c\in S(f)$ lie below all intersections of the line $x=c_x$ with 
$f[I]$. Then
$$
N(c,\,f)=1\,.
$$
To prove it, we consider 
the partition $a=a_0<a_1<\ds<a_k=b$ of the domain of $f$ and the components 
$z_1<z_2<\ds<z_l$ of 
$$
P(c,f)=f^{-1}[\ell(c)\cap f[I]\,]=
f^{-1}[(x=c_x)\cap f[I]\,]\,.  
$$
As we know, every component $z_i$ is an interval $[a_{r_i},a_{s_i}]$ ($\sus[a,b]$) with 
$0<r_i<s_i<k$ or a~point $\{t_i\}$ with $t_i\in (a,b)$. It is clear that 
$P(c,f)\ne\emptyset$ and $l\ge1$ (Proposition~\ref{prop_interValu}). We
call the $l+1$ open subintervals 
$$
g_0=(a_0,\min(z_1)),\, 
g_1=(\max(z_1),\min(z_2)),\,\ds,\,
g_l=(\max(z_l),\,,a_k)
$$
of $(a,b)$ {\em gaps}. It follows from  the assumption on $c$ that for every 
gap $g_i$ and every $t\in g_i$, all values $f(t)$ lie on only one side  of 
$\ell(c)$, the left side $L$ or the right side $R$.
We accordingly assign to $P(c,f)$ a~word $w_0w_1\ds w_l\in
\{L,R\}^{l+1}$. For $i=1,2,\ds,l$,  the transition $w_{i-1}w_i$ with 
$w_{i-1}=w_i$  
corresponds to a~double component of $P(c,f)$, and that with 
$w_{i-1}\ne w_i$ to a~simple component. 
Since $w_0\ne w_l$, the number of the latter transitions is odd.

\medskip\noindent
{\bf Step~3. Some two points in the complement of any circuit are separated. }Let $f\cc I\to\R^2$ be 
a~circuit and $D=\R^2\setminus f[I]$. We may assume that $f[I]$ lies to the 
right of the $y$-axis (Propositions~\ref{prop_imComp} and \ref{prop_CoClBo}). Let $l\in f[I]$ and $p\in f[I]$ be, respectively, 
a~leftmost and a~rightmost point in $f[I]$ (Propositions~\ref{prop_imComp} and \ref{prop_higLow}). These points divide $f$ in two arcs $f_1,f_2\cc I\to\R^2$ such that $f_1^0\cap f_2^0=\emptyset$, $f_1[I]\cup 
f_2[I]=f[I]$ and $\mathrm{ep}(f_1)=\mathrm{ep}(f_2)=\{l,p\}$ (Proposition~\ref{prop_delCykl}). Clearly, $l_x<p_x$ (Proposition~\ref{prop_CircNoSe}). 
Let $\ell$ be any vertical line with $x$-coordinate 
in $(l_x,p_x)$. Let $a\in f[I]\cap\ell$ be the highest point of this intersection
(Proposition~\ref{prop_higLow}). We 
may assume that $a\in f_1^0$. Similarly, let $b\in f_1^0\cap\ell$ 
be the lowest point of this intersection (Proposition~\ref{prop_higLow}). We take a~point $c\in\ell$ below $b$ but such that
$$
0<d_2(c,\,b)<d_2(b, f_2[I])
$$
---\,then $c\in D$. The point $c$ exists by Propositions~\ref{prop_imComp} and \ref{prop_distComSet}, 
by which $d_2(b, f_2[I])>0$. 
We show that for every PL map $g\cc I\to\R^2$ that joins $c\in D$ to the origin 
$\overline{0}=\langle0,0\rangle\in D$ ($f[I]$ lies to the right of the $y$-axis), $g[I]\cap f[I]\ne\emptyset$.
In view of Propositions~\ref{prop_compOpen} and \ref{prop_opArcCon} it means that 
the complement $D=\R^2\setminus f[I]$ is disconnected.

Suppose for the contrary that $g\cc I\to\R^2$ is a~PL map joining $c$ to 
$\overline{0}$ in $D$, so that $g[I]\cap f[I]=\emptyset$. Let 
$\kappa$ ($\sus\R^2$) be the half-line going down from $c$ and let
$$
\lambda=c\,b\cup f_3[I]\cup\ell(a)\ \ (\sus\R^2)\,,
$$
where $cb$ is the straight segment joining these two points, $f_3$ is 
a~subarc of $f_1$ joining $b$ to $a$ (if $b=a$, we omit $f_3[I]$) and, as 
before, $\ell(a)$ is the half-line going up from $a$. Let $h>0$ be the 
minimum of the three positive distances
$$
d_2(g[I],\,f[I])>0,\ d_2(\kappa,\,f_1[I])>0\,\text{ and }\,d_2(\lambda,\,f_2[I])>0\,.
$$
The first distance is positive by the assumption that $g[I]\cap f[I]=\emptyset$ and by 
Propositions~\ref{prop_imComp} and \ref{prop_distComSet}. The latter two distances are positive by 
Propositions~\ref{prop_uniClo}, \ref{prop_imagClos}, 
\ref{prop_imComp}, \ref{prop_CoClBo}, and \ref{prop_distComSet1}.

We approximate the given circuit $f\cc I=[\al,\be]\to\R^2$ sufficiently 
closely by a~closed PL map. By Proposition~\ref{prop_uniCon}, there 
exists a~partition 
$\al=a_0<a_1<\ds<a_k=\be$ of $[\al,\be]$ such that 
$\{a,b,l,p\}\sus
\{f(a_i)\cc\;i=0,1,\ds,k\}$ and
$$
d_2(f(a_{i-1}),\,f(a_i))<h\,\text{ for every }\,i=1,\,2,\,\ds,\,k\,.
$$
We join for $i=1,2,\ds,k$ the consecutive 
points $f(a_{i-1})$ and $f(a_i)$ ($f(a_0)=f(a_k)$)  by straight segments and obtain 
a~closed PL map $f_4\cc I=[\al,\be]\to\R^2$ such that 
all corners of $f_4[I]$ lie on $f[I]$ and include the four points $a$, $b$, 
$l$, and $p$. Also, for every $z\in f_4[I]$ we have $d_2(z,f[I])<
\frac{1}{2}h$. Let $f_5$, $f_6$, and $f_7$ be the non-closed PL sub-maps of 
$f_4$ spanned by the corners of $f_4[I]$ lying in $f_1[I]$, 
$f_2[I]$, and $f_3[I]$, respectively. 
The above definition of $h$ and the $\frac{1}{2}h$-approximation of 
$f_4[I]$ by $f[I]$ imply that
$$
f_4[I]\cap g[I]=\emptyset,\ 
f_5[I]\cap\kappa=\emptyset,\,\text{ and }\,f_6[I]\cap(cb\cup f_7[I]\cup\ell(a))=\emptyset\,.
$$

The parity function $N$ produces the contradiction
$$
0=N(\overline{0},\,f_4)=
N(c,\,f_4)=N(c,\,f_5)+N(c,\,f_6)=1+0
=1\,.
$$
The first equality follows from the fact that since $f[I]$ lies to the 
right of the $y$-axis, so does $f_4[I]$. The second equality follows 
from Step~1 because $f_4[I]\cap g[I]=\emptyset$. The third equality 
follows from the first property of $N$ in Step~2 because (we can define 
$f_5$ and $f_6$ in such a~way that) $f_4=f_5+f_6$. We have $N(c,f_5)=1$ 
by the second property of $N$ in Step~2 (since $f_5[I]\cap\kappa
=\emptyset$, the point $c$ lies below all intersections of the line $x=c_x$ with $f_5[I]$). Finally, $N(c,f_6)=N(a,f_6)$ 
by Step~1 because $(cb\cup f_7[I])\cap f_6[I]=\emptyset$, and $N(a,f_6)=0$ by the definition of $N$
because $\ell(a)\cap f_6[I]=\emptyset$.

\section{Concluding remarks}\label{sec_concl} 

Jordan's proof of his theorem in \cite{jord} was often questioned in the past, and it was claimed that the first correct proof of the JT was given by 
Veblen \cite{vebl}. Hales \cite{hale_jord} reversed these 
assessments and showed that Jordan's proof is basically correct when some gaps are 
filled, but that Veblen's proof has serious problems. Filippov's article 
\cite{fili} is preceded by the article \cite{volp} by A.~I. Vol'pert, which 
contains a similar proof of the JT. The footnote on the first page of 
\cite{fili} reads ``After this work was completed, the author learned 
that a~proof of Jordan's theorem, very close to the one in this article, 
was somewhat earlier found by A.~I. Vol'pert (Lviv)''. Authors of these two proofs, {\em 
Aleksei F. Filippov (1923--2006)} and {\em Aizik I. Volpert (Vol'pert) 
(1923--2006)}, were excellent mathematicians\,---\,see \cite{filiWiki,aizi}.

In recent years, the ways in which we (mathematicians and 
philosophers) think about the links between informal and formal 
proofs, and about the 
role of formalized mathematics, have changed considerably and dramatically 
compared to the times before, say, 2000: see, for example, 
\cite{brya_al,dvor,FLT,hale}. What should be the results of the work of 
a~research mathematician? A~possible answer is: rigorous proofs of 
mathematical theorems. An informal proof is rigorous if it can be routinely 
translated into a~formal proof (\cite{hama3}), formal in the sense of 
the previously mentioned developments; see also \cite{hama1,hama2}. In this 
article, we wanted to lift the beautiful proofs of Thomassen and 
Filippov to the modern level of rigor, and we hopefully succeeded. The 
preprint \cite{brya_al} reports on the process of autoformalization of 
Munkres' textbook \cite{munk} on topology. It is an impressive 
achievement; it seems that very soon AI will do everything for us. But does this 
autoformalization help a~student interested in topology learn and 
understand (\cite{hama_morr0,hama_morr}) the proofs of theorems in \cite{munk}? 
Not much. In contrast, we hope that our article will be 
helpful to anyone wishing to understand the two proofs of the WJT.   

I lectured on Thomassen's proof of the WJT in the course {\em Mathematical 
Analysis~3} in the years 2023--2026. For the next year, I plan to switch to 
the proofs of the JT in \cite{fili,volp}.

\end{document}